\documentclass[a4paper]{amsart}

\usepackage{microtype}

\usepackage{amssymb}

\usepackage{lmodern}
\usepackage{bbold}

\usepackage[hidelinks]{hyperref}

\usepackage{amsrefs}

\newcommand{\QQ}{\mathbb{Q}}
\newcommand{\ZZ}{\mathbb{Z}}
\newcommand{\NN}{\mathbb{N}}
\newcommand{\KK}{\mathbb{K}}
\newcommand{\MM}{\mathbb{M}}

\DeclareMathOperator{\N}{N}

\newcommand\set[1]{\left\{#1\right\}}
\newcommand\br[1]{\left(#1\right)}
\newcommand\abs[1]{\left|#1\right|}
\newcommand\bigO[1]{\mathop{O}\br{#1}}
\newcommand\bigOm[1]{\mathop{\Omega}\br{#1}}

\newtheorem{thm_main}{Theorem}
\newtheorem{conjecture}{Conjecture}
\newtheorem{thm}{Theorem}[section]
\newtheorem{lem}[thm]{Lemma}

\begin{document}

\title{On a conjecture of Levesque and Waldschmidt II}

\author[T. Hilgart]{Tobias Hilgart}
\email{tobias.hilgart@plus.ac.at}

\author[V. Ziegler]{Volker Ziegler}
\email{volker.ziegler@plus.ac.at}

\subjclass[2020]{11D25, 11D57}

\keywords{Simplest cubic fields, family of Thue equations, diophantine
equations}

\begin{abstract}
    Related to Shank's notion of simplest cubic fields, the family of parametrised Diophantine equations,
    \[
        x^3 - (n-1) x^2 y - (n+2) xy^2 - 1 = \br{ x - \lambda_0 y}\br{x-\lambda_1 y}\br{x - \lambda_2 y} = \pm 1,
    \]
    was studied and solved effectively by Thomas and later solved completely by Mignotte. 
    
    An open conjecture of Levesque and Waldschmidt \cite{lewa15} states that taking these parametrised Diophantine equations and twisting them not only once but twice, in the sense that we look at
    \[
        f_{n,s,t}(x,y) = \br{ x - \lambda_0^s \lambda_1^t y }\br{ x - \lambda_1^s\lambda_2^t y }\br{ x - \lambda_2^s\lambda_0^t y } = \pm 1,
    \]
    retains a result similar to what Thomas obtianed in the original or Levesque and Waldschidt in the once-twisted ($t = 0$) case; namely, that non-trivial solutions can only appear in equations where the parameters are small. We confirm this conjecture, given that the absolute values of the exponents $s, t$ are not too large compared to the base parameter $n$.
\end{abstract}

\maketitle

%%%%%%%%%%%%%%%%%%%%%%%%%%%%%%%%%%%%%%%%%%%%%%%%%%%%%%%%%%%%%%%
%
% Section -- Motivation and Statement
% 
%%%%%%%%%%%%%%%%%%%%%%%%%%%%%%%%%%%%%%%%%%%%%%%%%%%%%%%%%%%%%%%
\section{Motivation and Statement}
    A simplest cubic field, according to Shanks \cite{sha74}, is a cyclic (cubic) field, i.e. where the Galois action is a Moebius transform. They have nice arithmetic properties and, in particular, a fundamental system of units is explicitly known. Shanks constructed many such simplest cubic fields using the polynomials 
    \[
        f_n(x) = x^3-(n-1)x^2 -(n+2)x-1,
    \]
    whose discriminant is $(n^2+n+7)^2$, which in the case of primality of the inner term is also the discriminant of the associated number field. A fundamental system of untis for the number field $\KK_n = \QQ[X] / (f_n)$ is given by any pair of roots of $f_n$ (see \cite{sha74} or \cite{tho79}).

    Thomas \cite{tho90} considered a homogenised version of these polynomials,
    \[
        f_n(x,y) = f_n\br{\frac{x}{y}} y^3 = x^3 - (n-1) x^2y - (n+2) xy - y^3,
    \]
    and the equations $f_n(x, y) = \pm 1$, where $n \geq 0$. He effectively solved this infinite family of Thue equations, i.e. solved the equations where the parameter $n$ is at least $1.365 \times 10^7$, and was one of the first to obtain such a result; Mignotte \cite{mig93} later completed the work by also solving the cases where $ n < 1.365 \times 10^7$. A thorough investigation of such Thue equations and their correspondence to the simplest cubic fields, and in what sense the simplest cubic fields are unique, has been done by Hoschi~\cite{hosch11}, following unpublished work of Okazaki~\cite{oka} and building upon \cite{oka02}.

    Levesque and Waldschmidt \cite{lewa15} considered a twisted version of Thomas' equations. If $f_n(x, y)$ factorises to $(x - \lambda_0 \, y)(x - \lambda_1\, y)(x - \lambda_2 \, y)$, then they took each root to the $s$-th power and considered the equations
    \[
        f_{n, s}(x, y) = \br{ x - \lambda_0^s \, y } \br{ x - \lambda_1^s \, y } \br{ x - \lambda_2^s \, y } = \pm 1,
    \]
    now parametrised by integers $n \geq 0$ and $s \geq 1$. They effectively solved this infinite family of Thue equations and obtained a result similar to that of Thomas.

    Since the number fields $\KK_n$ involved in these equations are of unit rank $2$, and any two of the three roots $\lambda_0, \lambda_1, \lambda_2$ form a fundamental system of  units for $\KK_n$, Levesque and Waldschmidt conjectured \cite{lewa15}*{Conjecture~1.1} that a similar result to the one they obtained would hold if one considered the conjugates of $\lambda_0^s\lambda_1^t$ instead of $\lambda_0^s$; they formulated this in the following
    \begin{conjecture}\label{conj: lewa}
        There exists a positive absolute constant $\kappa$ with the following porperty: If $n, s, t, x, y, m$ are integers satisfying
        \[
            \max\set{\abs{x}, \abs{y}} \geq 2, \quad (s,t) \neq (0,0) \; \text{ and } \; 0 < \abs{f_{n,s,t}(x,y)} \leq m,
        \]
        then
        \[
            \max\set{ \log\abs{n}, \abs{s}, \abs{t}, \log\abs{x}, \log\abs{y} } \leq \kappa \br{ 1 + \log m }.
        \]
    \end{conjecture}
    
    For $m=1$ the conjecture recovers the direct generalisation ($t$ can be something other than $0$) of their generalisation ($s$ can be something other than $1$) of Thomas' original result. Another way of looking at these equations is to take any unit $\rho \in \ZZ_{\KK_n}^\times$ with conjugates $\rho_2$ and $\rho_3$ and the equations
    \[
        f_{n, \rho}(x,y) = \br{ x - \rho \, y }\br{ x - \rho_2 \, y }\br{ x - \rho_3 \, y } = \pm 1,
    \]
    and range over all integers $n$ and units $\rho \in \ZZ_{\KK_n}^\times$ instead of integers $s$ and $t$.

    The original result of Levesque and Waldschmidt \cite{lewa15}*{Theorem~1.1} proves Conjecture~\ref{conj: lewa} in the case that $t=0$ and $m=1$. For $m>1$ they could not give an upper bound for $\abs{n}$. The "$t=0$" case of Conjecture~\ref{conj: lewa} is already known in this sense, and therefore the focus of our work is to extend it. Our result is as follows:

    \begin{thm_main}\label{thm: main}
        Let $f_n(x) = x^3 - (n-1) x^2 - (n+2) x - 1$ and $\lambda_0, \lambda_1, \lambda_2$ be its roots in the number field $\KK_n = \QQ[X]/(f_n)$. Define the norm-form
        \[
            f_{n, s, t}(x, y) = \N_{\KK_n / \QQ}\br{ x - \lambda_0^s \lambda_1^t y } = \br{ x - \lambda_0^s \lambda_1^t y }\br{ x - \lambda_1^s\lambda_2^t y }\br{ x - \lambda_2^s\lambda_0^t y },
        \]
        and consider the parametrised family of Thue equations
        \[
            f_{n, s, t}(x, y) = \pm 1,
        \]
        where $n \in \NN$, $(s,t) \in \ZZ^2$ with $st \neq 0$ and $\max\set{\abs{s}, \abs{t}} \leq n^{1/2 - \varepsilon} $ for some fixed $0 < \varepsilon < \frac{1}{2}$. Then there exists an effectively computable constant $n_0 = n_0(\varepsilon)$, such that for any $n > n_0$ and $(s, t)$ as above the Thue equation $f_{n, s, t}(x,y) = \pm 1$ has no solutions where $\abs{y} \geq 2$.
    \end{thm_main}

    While Levesque and Waldschmidt \cite{lewa15}*{p.~540--541} were able to trace the case of negative parameters back to the case where $n\geq 0$ and $s\geq 1$ due to some symmetries involved , we have not been able to confirm any such symmetries in our equations and thus have to deal with both positive and negative exponents.

    For our proof, we were inspired by \cite{htz04}, which describes an algorithmic procedure for solving (polynomially, single) parametrised Thue equations. Of course, our starting point is different, since we have not only multiple parameters, but also parameters in the exponents. In Section~\ref{sec: lemmas}, we ensure that with a different starting situation and even some necessary conditions in the original paper that do not hold here, their idea is still a valid strategy and we can perform a proof similar to the one they describe.

    In the rest of the paper, we use the $O$ [$\Omega$] notation to describe asymptotic upper [lower] bounds for $n\to \infty$; the implied constants depend at most only on $\varepsilon$.

%%%%%%%%%%%%%%%%%%%%%%%%%%%%%%%%%%%%%%%%%%%%%%%%%%%%%%%%%%%%%%%
%
% Section -- Auxilary Results
% 
%%%%%%%%%%%%%%%%%%%%%%%%%%%%%%%%%%%%%%%%%%%%%%%%%%%%%%%%%%%%%%%
\section{Auxiliary Results}
\label{sec: lemmas}

    In our proof we derive a lower bound for $\log\abs{y}$, which we then compare with the following upper bound due to Bugeaud and Gy\H{o}ry \cite{bugy96}:
    \begin{thm}\label{thm:bugy}
        Let $F(X, Y) \in \ZZ[X, Y]$ be an irreducible binary form of degree $n \geq 3$ and let $b$ be a non-zero rational integer with absolute value at most $B (\geq e)$. Let $\MM = \QQ(\alpha)$ for some zero $\alpha$ of $F(X, 1)$, and denote by $R_\MM$ the regulator, and by $r = r_\MM$ the unit rank of $\MM$. Further, let $H (\geq 3)$ be an upper bound for the height of $F$. Then all solutions $x, y$ of the equation
        \[
            F(x, y) = b \qquad \text{ in } x, y \in \ZZ
        \]
        satisfy
        \[
            \max\set{\abs{x}, \abs{y}} < \exp\set{ c_3 R_\MM \max\set{\log R_\MM, 1} ( R_\MM + \log(HB) ) },
        \]
        where
        \[
            c_3 = c_3(n, r) = 3^{r+27} (r+1)^{7r+19} n^{2n+6r+14}.
        \]
    \end{thm}

    The roots $\lambda_0, \lambda_1, \lambda_2$ of the polynomial $f_n$ also play a fundamental role in our proof. Therefore, and for the sake of simplicity, we summarise some known properties in the following lemmas.

    \begin{lem}\label{lem: lshuffle}
        Let $\lambda_0$ be a root of the polynomial $f_n(x) = x^3 - (n-1) x^2 - (n+2) x - 1$, then
        \[
            \lambda_1 = - \cfrac{1}{\lambda_0 + 1}, \;\;\; \lambda_2 = - \cfrac{\lambda_0 + 1}{\lambda_0}
        \]
        are the other two roots.
    \end{lem}
    \begin{proof}
        Since $\lambda_0 \lambda_1 \lambda_2 = 1$, it suffices to prove that $f_n\br{-\frac{1}{\lambda_0 + 1}} = 0$, which is a straightforward calculation, using $\lambda_0^3 = (n-1) \lambda_0^2 + (n+2)\lambda_0 + 1$.
    \end{proof}

    \begin{lem}\label{lem: lapprox}
        Let $\lambda_0, \lambda_1, \lambda_2$ be the roots of the polynomial $x^3 - (n-1) x^2 - (n+2) x - 1$, then
        \begin{alignat*}{11}
            \lambda_0 &= &&n &&+ \frac{2}{n} &&+ \bigO{n^{-2}}, \;\;\;  &&\log\abs{\lambda_0} &&= &&\phantom{+}\log n && && &&+ \frac{2}{n^2} &&+ \bigO{n^{-3}} \\
            \lambda_1 &= -&&\frac{1}{n} &&+ \frac{1}{n^2} &&+ \bigO{n^{-3}}, \;\;\; &&\log\abs{\lambda_1} &&= &&-\log n &&-&&\frac{1}{n} &&-\frac{3}{2n^2} &&+ \bigO{n^{-3}} \\
            \lambda_2 &= -&&1 &&- \frac{1}{n} &&+ \bigO{n^{-3}}, \;\;\; &&\log\abs{\lambda_2} &&= && &&\phantom{+}&&\frac{1}{n} &&- \frac{1}{2n^2} &&+ \bigO{n^{-3}}.
        \end{alignat*}
    \end{lem}
    \begin{proof}
        Each asymptotic can be easily verified, e.g. $f\br{ n + \frac{2}{n} \pm \frac{3}{n^2} }$ is positive and negative, respectively and thus $\lambda_0 = n + \frac{2}{n} + \bigO{n^{-2}}$ by the intermediate value theorem. Taking the logarithm gives $\log\lambda_0 = \log n + \log\br{1 + \frac{2}{n^2} + \bigO{n^{-3}} }$, and calculating the first term in the Taylor-expansion proves the assertion for $\log\lambda_0$. The results for $\lambda_1$ and $\lambda_2$ follow analogously.
    \end{proof}

    We know (\cite{sha74} or \cite{tho79}) that every pair out of $\lambda_0, \lambda_1, \lambda_2$ gives a fundamental system of units for the number field $\KK = \QQ(\lambda_0)$. Thus the regulator $R$ of $\KK$ can be calculated by the the determinant
    \[
        R =
        \begin{vmatrix}
            \log\abs{\lambda_1} & \log\abs{\lambda_2} \\
            \log\abs{\lambda_2} & \log\abs{\lambda_0}
        \end{vmatrix};
    \]
    by the previous lemma, we have the following asymptotic for it:
    \begin{lem}\label{lem: regulatorapprox}
        \[
            R = (\log n)^2 + \frac{\log n}{n} + \bigO{\frac{\log n}{n^2}}.
        \]
    \end{lem}

    We call a solution $(x,y)$ of the equation $f_{n,s,t}(x,y)$ of type $j$, if for $\alpha^{(1)} = \lambda_0^s\lambda_1^t, \alpha^{(2)} = \lambda_1^s\lambda_2^t, \alpha^{(3)} = \lambda_2^s\lambda_0^t$
    \[
        \abs{x - \alpha^{(j)} y} = \min\set{ \abs{x - \alpha^{(i)} y} : 1\leq i \leq 3 }.
    \]

    The following result allows us to consider only solutions of type 1.

    \begin{lem}\label{lem: onlytype1}
        Let $(x,y)$ be a solution of type $2$ or $3$ of $f_{n, s, t}(x,y) = \pm 1$, then it is a solution of type $1$ of $f_{n, -s+t,-s}(x, y) = \pm 1$ or $f_{n, -t, s-t}(x, y) = \pm 1$, respectively.
    \end{lem}
    \begin{proof}
        We use Lemma~\ref{lem: lshuffle} to express each $\alpha^{(i)}$ as powers of $\lambda_0, \lambda_1$, i.e
        \begin{align*}
            \alpha^{(1)} = \lambda_0^s\lambda_1^t, \quad
            \alpha^{(2)} = \lambda_0^{-t} \lambda_1^{s-t}, \quad
            \alpha^{(3)} = \lambda_0^{-s+t} \lambda_1^{-s},
        \end{align*}
        then the bijective linear transformation $\phi: (s, t) \mapsto (-s+t, -s)$ indeed maps 
        \[
            \alpha^{(1)} \mapsto \alpha^{(3)} \mapsto \alpha^{(2)} \mapsto \alpha^{(1)}.
        \]
    \end{proof}

    \begin{lem}\label{lem: lpowers}
        Let $a \in \ZZ$ with $\abs{a} = \bigO{n^{1/2 - \varepsilon}}$, then
        \begin{alignat*}{4}
            \lambda_0^a &= n^a &&+ 2a n^{a-2} &&+ \bigO{ n^{a-2-2\varepsilon} } \\
            (-1)^a\lambda_1^a &= n^{-a} &&- a n^{-a-1} &&+ \bigO{ n^{-a-1-2\varepsilon} } \\
            (-1)^a\lambda_2^a &= 1 &&+ a n^{-1} &&+ \bigO{ n^{-1-2\varepsilon} }.                
        \end{alignat*}
    \end{lem}
    \begin{proof}
        The result follows from Lemma~\ref{lem: lapprox} and Taylor-expansion.
    \end{proof}

    \begin{lem}\label{lem: logdiffapprox}
        We have, up to an error of order $\bigO{n^{-1-2\varepsilon}}$, that $\log\abs{\alpha^{(1)} - \alpha^{(2)}}$ is either
        \[
            \left\{
            \begin{aligned}
                (s-t)\cdot &\log n & & &- &\cfrac{t}{n} &&\text{ if } 2s > t+1  \\
                (s-t)\cdot &\log n & & &- &\cfrac{t - (-1)^s}{n} &&\text{ if } 2s = t+1 \\
                (s-t)\cdot &\log n &+ &\log 2 &+ &\cfrac{s-2t}{2n}  &&\text{ if } 2s = t, s \text{ odd} \\
                (s-t-1)\cdot &\log n &+ &\log\abs{s} &\phantom{+} & &&\text{ if } 2s = t, s \text{ even} \\
                (-s)\cdot &\log n & & &+ &\cfrac{s-t-(-1)^s}{n} &&\text{ if } 2s = t-1 \\
                (-s)\cdot &\log n & & &+ &\cfrac{s-t}{n} &&\text{ if } 2s < t-1,
            \end{aligned}
            \right.
        \]
        and $\log\abs{ \alpha^{(1)} - \alpha^{(3)} }$ is either
        \[ 
            \left\{
            \begin{aligned}
                (s-t)\cdot &\log n & & &- &\cfrac{t}{n} &&\text{ if } s > 2t+1 \\
                (s-t)\cdot &\log n & & &- &\cfrac{t + (-1)^t}{n}  &&\text{ if } s = 2t+1 \\
                t\cdot &\log n &+ &\log 2 &+ &\cfrac{s-2t}{2n}  &&\text{ if } s = 2t, t \text{ even} \\
                (t-1)\cdot &\log n &+ &\log\abs{s+t} &\phantom{+} & &&\text{ if } s = 2t, t \text{ odd} \\
                t\cdot &\log n & & &+ &\cfrac{s-(-1)^t}{n}  &&\text{ if } s = 2t-1 \\
                t\cdot &\log n & & &+ &\cfrac{s}{n}  &&\text{ if } s < 2t-1.
            \end{aligned}
            \right.
        \]        
    \end{lem}
    \begin{proof}
        We express each of the algebraic numbers $\alpha^{(1)}, \alpha^{(2)}$ and $\alpha^{(3)}$ as powers of $\lambda_0$ and $\lambda_2$. Lemma~\ref{lem: lpowers} and case-differentiation on which terms are of higher or equal order and taking the logarithms yields the result.

        For example, if we consider the difference $\abs{\alpha^{(1)}-\alpha^{(2)}}$, we can write this as
        \[
            \abs{ \lambda_0^{s-t} \lambda_2^{-t} - \lambda_0^{-s} \lambda_2^{-s+t} }
        \]
        using Lemma~\ref{lem: lshuffle}. By Lemma~\ref{lem: lpowers}, we can further write $\lambda_0^{s-t}\lambda_2^{-t}$ as
        \begin{align*}
            &(-1)^t \br{ n^{s-t} + 2(s-t)n^{s-t-2} + \bigO{n^{s-t-2-2\varepsilon}} } \br{ 1 - tn^{-1} + \bigO{n^{-1-2\varepsilon}} } \\
            = &(-1)^t \br{ n^{s-t} - t n^{s-t-1} + \bigO{n^{s-t-1-2\varepsilon}} }
        \end{align*}
        and $\lambda_0^{-s}\lambda_2^{-s+t}$ as
        \begin{align*}
            &(-1)^{-s+t} \br{ n^{-s} - 2s n^{-s-2} + \bigO{n^{-s-2-2\varepsilon}} } \br{ 1 - (s-t)n^{-1} + \bigO{n^{-1-2\varepsilon}} }. \\
            =&(-1)^{-s+t} \br{ n^{-s} - (s-t) n^{-s-1} + \bigO{ n^{-s-1-2\varepsilon} } }.
        \end{align*}

        Then we just have to distinguish which of the four explicit terms dominates the rest, e.g. if $2s > t+1$, then $n^{s-t}$ and $-tn^{s-t-1}$ dominate $n^{-s}$ and $-(s-t)n^{-s-1}$, and the latter two terms can even be shifted into the $\bigO{ n^{s-t-1-2\varepsilon} }$. Factoring $n^{s-t}$ and expanding the logarithm of $1 - tn^{-1} + \bigO{n^{-1-2\varepsilon}}$ finally gives
        \[
            \log\abs{ \alpha^{(1)} - \alpha^{(2)} } = (s-t) \log n - \frac{t}{n} + \bigO{n^{-1-2\varepsilon}}.
        \]
    \end{proof}

    \begin{lem}\label{lem: errorbound}
        Let $n \geq n_0$ be sufficiently large, then we have
        \[
            \abs{ \alpha^{(1)} - \alpha^{(2)} } \cdot \abs{ \alpha^{(1)} - \alpha^{(3)} } > \frac{2}{3}\, n^2
        \]
        for $(s, t) \not\in \set{ (-1, -1), (1, 1) }$. The bounds
        \begin{align*}
            \min\set{ \abs{ \alpha^{(1)} - \alpha^{(2)} }^2 \cdot \abs{ \alpha^{(1)} - \alpha^{(3)} }, \abs{ \alpha^{(1)} - \alpha^{(2)} } \cdot \abs{ \alpha^{(1)} - \alpha^{(3)} }^2 } &> \frac{2}{3}\, n, \\
            \max\set{ \abs{ \alpha^{(1)} - \alpha^{(2)} }^2 \cdot \abs{ \alpha^{(1)} - \alpha^{(3)} }, \abs{ \alpha^{(1)} - \alpha^{(2)} } \cdot \abs{ \alpha^{(1)} - \alpha^{(3)} }^2 } &> \frac{2}{3}\, n^2
        \end{align*}
        hold for all $(s, t)$ with $st \neq 0$.
    \end{lem}
    \begin{proof}
        The choice of the constant $\frac{2}{3}$ is somewhat arbitrary. Any constant $<1$ would be correct if $n$ is sufficiently large, and for our purposes we need a constant (explicitly) larger than $\frac{1}{2}$.
    
        The stated result follows from either Lemma~\ref{lem: lpowers} or Lemma~\ref{lem: logdiffapprox} and case-differentiation; in each case, the exponent of the leading term is at least $2$ and $1$, respectively. 
        
        We demonstrate the case where $2s \leq t-1$ and $s \geq 2t+1$. From Lemma~\ref{lem: logdiffapprox} we can see that the power of $n$ in $\abs{\alpha^{(1)} - \alpha^{(2)}}$ and $\abs{\alpha^{(1)} - \alpha^{(3)}}$ is $-s$ and $s-t$, respectively. So the power of $n$ in the product is thus $-t$, and in the case where $2s \leq t-1$ and $s \geq 2t+1$, we have $-t \geq 2$ unless $s = t = -1$.
        If instead we square one of the terms and then take the product, the powers of $n$ are $-s-t$ and $s-2t$, respectively. We also have that $-s-t \geq 2$ and $s-2t \geq 1$.
    \end{proof}

    \begin{lem}\label{lem: vbar}
        Let
        \[
            v_1 =
            \begin{vmatrix}
                \log\abs{ \alpha^{(1)} - \alpha^{(2)} } & \log\abs{ \lambda_2 } \\ \\
                \log\abs{ \alpha^{(1)} - \alpha^{(3)} } & \log\abs{ \lambda_0 }
            \end{vmatrix},
            \qquad
            v_2 =
            \begin{vmatrix}
                \log\abs{ \lambda_1 } & \log\abs{ \alpha^{(1)} - \alpha^{(2)} } \\ \\
                \log\abs{ \lambda_2 } & \log\abs{ \alpha^{(1)} - \alpha^{(3)} }
            \end{vmatrix},            
        \]
        then there exists an integer $b_0$ such that for $\Bar{v} = b_0 R - v_1 - v_2$, we have
        \[
            \Bar{v} \geq \cfrac{\log n}{n} + \bigO{ \frac{\log n}{ n^{1+2\varepsilon} } }, \;\; R - \Bar{v} = \bigOm{ \log n }.
        \]
    \end{lem}
    \begin{proof}
        By Lemma~\ref{lem: lapprox}, we have
        \begin{align}\label{eq: v1v2}
            -v_1 - v_2 = \;\;\; &\log\abs{ \alpha^{(1)} - \alpha^{(2)} } \br{ -\log n + \frac{1}{n} + \bigO{ n^{-2} } } + \\
            &\log\abs{ \alpha^{(1)} - \alpha^{(3)} } \br{ \log n + \frac{2}{n} + \bigO{n^{-2}} }, \nonumber
        \end{align}
        and we use Lemma~\ref{lem: logdiffapprox} to verify the statement in each case. 
        
        We start with the special cases $2s=t$ and $s=2t$, then $-v_1 - v_2$ is of the form
        \[
            -v_1 - v_2 = -l\, \br{ \log n }^2 + \xi \log n + \bigO{ \frac{\log n}{ n^{1/2 + \varepsilon} } },
        \]
        where $ \xi \in \set{ -\log 2, -\log\abs{s} }$ and $\xi \in \set{\log 2, \log\abs{s+t}}$, respectively, and $l$ is some linear combination of $1, s, t$ the precise shape of which we are not interested in. If $2s =t$, then taking $b_0 = l+1$ gives
        \[
            \Bar{v} = c_0\, R - v_1 - v_2 = \br{ \log n }^2 + \xi \log n + \bigO{ \frac{\log n}{n^{1+2\varepsilon}} },
        \]
        and since $\xi \in \set{-\log 2, - \log\abs{s}}$ and $\abs{s} \leq n^{\frac{1}{2}-\varepsilon}$, we have 
        \[
            \xi \geq -\log\abs{ n^{ \frac{1}{2} - \varepsilon } } + \bigO{1} = \br{ - \frac{1}{2} + \varepsilon } \log n + \bigO{1}.
        \]
        
        The term $\br{\log n}^2$ in $\Bar{v}$ can thus at worst be $\br{ \frac{1}{2} + \varepsilon} \br{\log n}^2$, when taking $\xi \log n$ into consideration, i.e.
        \[
            \Bar{v} = \bigOm{ \br{\log n}^2 },
        \]
        and by Lemma~\ref{lem: regulatorapprox}, we have
        \[
            R - \Bar{v} = -\xi \log n + \bigO{ \frac{\log n}{n^{1/2 + \varepsilon}} } = \bigOm{ \log n }, 
        \]
        as required.

        If $s = 2t$, we instead take $b_0 = l$, then $\Bar{v} = \xi \log n + \bigO{ \frac{ \log n }{ n^{ 1/2 + \varepsilon } } } = \bigOm{ \log n }$, since $\xi > 0$. And we also have $ R - \Bar{v} = \br{ \log n }^2 - \xi \log n + \bigO{ \frac{\log n}{ n^{ 1/2 + \varepsilon } } } = \bigOm{ \br{\log n}^2 } $, which is even better than what we require.

        In the other cases, we do not have the $\xi \log n$ term and have to pay more attention to the other expressions; we will demonstrate the case where both $2s \leq t-1$ and $s \geq 2t+1$ hold, other cases can be proved similarly.
        
        If we let $\mathbb{1}_a(b)$ be $1$ if $a = b$ and $0$ otherwise, then by Lemma~\ref{lem: logdiffapprox}, we have that
        \begin{alignat*}{5}
            \log\abs{ \alpha^{(1)} - \alpha^{(2)} } &= \phantom{(t}(-s) &\cdot \log n &+ \cfrac{ s - t - (-1)^s \mathbb{1}_{2s}(t-1) }{n} &&+ \bigO{n^{-1-2\varepsilon}}, \\
            \log\abs{ \alpha^{(1)} - \alpha^{(3)} } &= (s-t) &\cdot \log n &- \cfrac{ t + (-1)^t \mathbb{1}_{s}(2t+1) }{n} &&+ \bigO{n^{-1-2\varepsilon}}.
        \end{alignat*}

        We plug this into Equation~\eqref{eq: v1v2} and get
        \[
            -v_1 - v_2 = \br{2s-t} \br{ \log n }^2 + \br{ -2t + \xi' } \cfrac{\log n}{n} + \bigO{ \frac{\log n}{n^{1+2\varepsilon}} },
        \]
        where $\xi' \in \set{0, \pm 1, \pm 2}$, depending on which of the indicators $\mathbb{1}_{2s}(t-1)$ and $\mathbb{1}_{s}(2t+1)$ is non-zero, if any is at all; we only care that $\xi' \geq -2$.

        Putting $b_0 = -2s+t$ then gives
        \[
            \Bar{v} = \br{-2s+t} R - v_1 - v_2 = \br{ -2s - t + \xi' } \cfrac{\log n}{n} + \bigO{ \frac{\log n}{n^{1+2\varepsilon}} }.
        \]

        The system of linear inequalities
        \begin{align*}
            2s &\leq t-1 \\
            s &\geq 2t + 1 \\
            -2s-t -2 &\leq 0
        \end{align*}
        has no solutions, and thus indeed $-2s-t+\xi' \geq 1$, or $\Bar{v} \geq \cfrac{\log n}{n} + \bigO{ \frac{\log n}{n^{1+2\varepsilon}} } $. Subtracting it from the regulator also gives $R-\Bar{v} = \bigOm{ \br{\log n}^2 }$, which is even more than required.
    \end{proof}

%%%%%%%%%%%%%%%%%%%%%%%%%%%%%%%%%%%%%%%%%%%%%%%%%%%%%%%%%%%%%%%
%
% Section -- Proof of Theorem
% 
%%%%%%%%%%%%%%%%%%%%%%%%%%%%%%%%%%%%%%%%%%%%%%%%%%%%%%%%%%%%%%%
\section{Proof of Theorem~\ref{thm: main}}
    
    We assume that for sufficiently large $n$ there exists a solution $(x,y)$ with $\abs{y} \geq 2$ of the Thue equation $f_{n,s,t}(x,y) = \pm 1$ and we want to derive a contradiction. We recall that we denote by $\alpha^{(1)}, \alpha^{(2)}, \alpha^{(3)}$ the algebraic elements in the norm-form in the order in which they appear in the theorem, and write $\beta^{(i)} = x - \alpha^{(i)} y$. A solution $(x,y)$ of the Thue equation $f_{n, s, t}(x, y) = \pm 1$ is then of type $j$, if
    \[
        \abs{ \beta^{(j)} } = \min\set{ \abs{\beta^{(1)}}, \abs{\beta^{(2)}}, \abs{\beta^{(3)}} }.
    \]
    
    By Lemma~\ref{lem: onlytype1}, we only have to consider $j = 1$. In this case, we have
    \[
        2\abs{ x - \alpha^{(i)} y } \geq \abs{ x - \alpha^{(i)} y } + \abs{ x - \alpha^{(1)} y } \geq \abs{ y \br{ \alpha^{(1)} - \alpha^{(i)} } },
    \]
    for $i \in \set{2, 3}$. Applying the inequality to $\abs{ f_{n, s, t}(x, y) } = 1$ yields
    \begin{equation}\label{eq: betajbound}
        \abs{ \beta^{(1)} } \leq \cfrac{ 4 }{ \abs{y}^2 \abs{ \alpha^{(1)} - \alpha^{(2)} } \abs{\alpha^{(1)} - \alpha^{(3)}} } \leq \cfrac{1}{ \abs{ \alpha^{(1)} - \alpha^{(2)} } \abs{\alpha^{(1)} - \alpha^{(3)}} }.
    \end{equation}
    
    We express $\log\abs{\beta^{(2)}}$ and $\log\abs{\beta^{(3)}}$ as
    \begin{align*}
        \log\abs{ \beta^{(i)} } &= \log\abs{ x - \alpha^{(1)} y + y\br{ \alpha^{(1)} - \alpha^{(i)} } } \\
        &= \log\abs{y} + \log\abs{ \alpha^{(1)} - \alpha^{(i)} } + \log\abs{ 1 + \cfrac{ \beta^{(1)} }{ y\br{ \alpha^{(1)} - \alpha^{(i)} } } },
    \end{align*}
    and apply both Inequality~\eqref{eq: betajbound} and Lemma~\ref{lem: errorbound} on the argument of the last logarithm. This gives
    \[
        \abs{ \cfrac{ \beta^{(1)} }{ y\br{ \alpha^{(1)} - \alpha^{(i)} } } } < \cfrac{1}{n},
    \]
    which is less than $1$. And since we have $\log(1+a) = a + \bigO{a^2}$ for $\abs{a} < 1$, this implies
    \begin{equation}\label{eq: betaiexpr}
        \log\abs{ \beta^{(i)} } = \log\abs{y} + \log\abs{ \alpha^{(1)} - \alpha^{(i)} } + \cfrac{ \beta^{(1)} }{ y\br{ \alpha^{(1)} - \alpha^{(i)} } } + \bigO{ n^{-2} }.
    \end{equation}
    
    Since $\set{\lambda_0, \lambda_1}$ is a fundamental system of units in $\KK$, there exist $b_1, b_2 \in \ZZ$ such that $\beta^{(1)} = \pm \lambda_0^{b_1} \lambda_1^{b_2}$. By conjugating, we get that
    \begin{align*}
        \log\abs{ \beta^{(2)} } &= b_1 \log\abs{ \lambda_1 } + b_2 \log\abs{ \lambda_2 }, \\
        \log\abs{ \beta^{(3)} } &= b_1 \log\abs{ \lambda_2 } + b_2 \log\abs{ \lambda_0 }.
    \end{align*}

    To this system of linear equations in $b_1$ and $b_2$, we want to apply Cramer's rule and use Equation~\eqref{eq: betaiexpr} to express $\log\abs{\beta^{(2)}}$ and $\log\abs{\beta^{(3)}}$. We have already given an asymptotic for the determinant of the matrix in this system, which is precisely the regulator $R$ (up to sign), in Lemma~\ref{lem: regulatorapprox}.
    
    Cramer's rule then gives
    \begin{align*}
        R\, b_1 &= u_1 \log\abs{y} + v_1 + \cfrac{ \beta^{(1)} }{ y }\, w_1 + \bigO{ \frac{ \log n }{ n^2 } }, \\
        R\, b_2 &= u_2 \log\abs{y} + v_2 + \cfrac{ \beta^{(1)} }{ y }\, w_2 + \bigO{ \frac{ \log n }{ n^2 } },
    \end{align*}
    with
    \begin{alignat*}{3}
        u_1 &=
        \begin{vmatrix}
            1 & \log\abs{ \lambda_2 } \\ \\
            1 & \log\abs{ \lambda_0 }
        \end{vmatrix},
        \;
        v_1 &=
        \begin{vmatrix}
            \log\abs{ \alpha^{(1)} - \alpha^{(2)} } & \log\abs{ \lambda_2 } \\ \\
            \log\abs{ \alpha^{(1)} - \alpha^{(3)} } & \log\abs{ \lambda_0 }
        \end{vmatrix},
        \;
        w_1 &=
        \begin{vmatrix}
            \cfrac{1}{ \alpha^{(1)} - \alpha^{(2)} } & \log\abs{ \lambda_2 } \\
            \cfrac{1}{ \alpha^{(1)} - \alpha^{(3)} } & \log\abs{ \lambda_0 }
        \end{vmatrix},
        \\
        u_2 &=
        \begin{vmatrix}
            \log\abs{ \lambda_1 } & 1 \\ \\
            \log\abs{ \lambda_2 } & 1
        \end{vmatrix},
        \;
        v_2 &=
        \begin{vmatrix}
            \log\abs{ \lambda_1 } & \log\abs{ \alpha^{(1)} - \alpha^{(2)} } \\ \\
            \log\abs{ \lambda_2 } & \log\abs{ \alpha^{(1)} - \alpha^{(3)} }
        \end{vmatrix},
        \;
        w_2 &=
        \begin{vmatrix}
            \log\abs{ \lambda_1 } & \cfrac{1}{ \alpha^{(1)} - \alpha^{(2)} } \\
            \log\abs{ \lambda_2 } & \cfrac{1}{ \alpha^{(1)} - \alpha^{(3)} }
        \end{vmatrix}.
    \end{alignat*}

    By Lemma~\ref{lem: lapprox}, the linear combination $\Bar{u} = -u_1-u_2$ is a small positive number, namely
    \begin{equation}\label{eq: ubar}
        \Bar{u} = - \br{ \log n - \frac{1}{n} + \bigO{n^{-2}} } - \br{ -\log n - \frac{2}{n} + \bigO{n^{-2}} } = \frac{3}{n} + \bigO{n^{-2}}.
    \end{equation}

    We take the same linear combination $\Bar{w} = -w_1 - w_2$, and for the integer $b_0$ from Lemma~\ref{lem: vbar}, we define $\Bar{v} = b_0 R - v_1 - v_2$ and $\Bar{b} = b_0 - b_1 - b_2$. With these choices we still have
    \begin{equation}\label{eq: cramer}
        R\, \Bar{b} = \Bar{u}\, \log\abs{y} + \Bar{v} + \frac{ \beta^{(1)} }{y}\, \Bar{w} + \bigO{ \frac{ \log n }{ n^2 } }.
    \end{equation}

    By Inequality~\eqref{eq: betajbound} and $\abs{y} \geq 2$, we have $\abs{ \cfrac{ \beta^{(1)} }{ y }\, \Bar{w} } \leq \abs{ \cfrac{ \Bar{w} }{2 \abs{ \alpha^{(1)} - \alpha^{(2)} } \abs{ \alpha^{(1)} - \alpha^{(3)} } } }$. First, we use Lemma~\ref{lem: lapprox} to write $\Bar{w}$ as
    \begin{align*}
        \Bar{w} &=\cfrac{1}{ \alpha^{(1)} - \alpha^{(2)} } \br{ -\log\abs{\lambda_0} + \log\abs{\lambda_2} } + \cfrac{1}{ \alpha^{(1)} - \alpha^{(3)} } \br{ -\log\abs{\lambda_1} +  \log\abs{\lambda_2} } \\
        &= - \cfrac{ \log n }{ \alpha^{(1)} - \alpha^{(2)} } + \cfrac{ \log n }{ \alpha^{(1)} - \alpha^{(3)} } + \bigO{ \frac{1}{n \min\set{ \abs{ \alpha^{(1)} - \alpha^{(2)} }, \abs{ \alpha^{(1)} - \alpha^{(2)} } } }}.
    \end{align*}

    We plug this into the inequality
    \[
        \abs{ \cfrac{ \beta^{(1)} }{ y }\, \Bar{w} } \leq \abs{ \cfrac{ \Bar{w} }{2 \abs{ \alpha^{(1)} - \alpha^{(2)} } \abs{ \alpha^{(1)} - \alpha^{(3)} } } }
    \]
    and apply Lemma~\ref{lem: errorbound}, which gives
    \begin{equation}\label{eq: wbar}
        \abs{ \cfrac{ \beta^{(1)} }{ y }\, \Bar{w} } < \frac{3}{4}\cfrac{\log n}{n} + \bigO{ \frac{ \log n }{ n^2 } };
    \end{equation}
    the constant in Lemma~\ref{lem: errorbound} ensures that the constant $\frac{3}{4}$ of $\frac{\log n}{n}$ is smaller than $1$, which allows the next arguments.
    
    Using this bound in conjunction with the asymptotic for $\Bar{u}$ of Equation~\eqref{eq: ubar}, we can thus derive from Equation~\eqref{eq: cramer} that
    \[
        R\, \Bar{b} > \cfrac{3 \log\abs{y}}{n} + \Bar{v} - \frac{3}{4}\cfrac{\log n}{n} + \bigO{ \frac{\log n}{n^2} },
    \]
    and since by Lemma~\ref{lem: vbar}, we have $\Bar{v} \geq \frac{\log n}{n} + \bigO{ \frac{\log n}{n^{1+2\varepsilon}} }$, this implies that
    \[
        R\, \bar{b} > \frac{3 \log\abs{y}}{n} + \br{1 - \frac{3}{4}}\frac{\log n}{n} + \bigO{ \frac{\log n}{n^2} } > 0.
    \]
    
    Since $\Bar{b}$ is an integer, it must be at least $1$. Equations~\eqref{eq: ubar} and \eqref{eq: cramer} thus yield
    \[
        R - \Bar{v} - \cfrac{ \beta^{(1)} }{ y }\, \Bar{w} \leq \cfrac{3 \log\abs{y}}{n} + \bigO{ \frac{\log n}{n^2} },
    \]
    and the left-hand side, by Inequality~\eqref{eq: wbar} and Lemma~\ref{lem: vbar} is of order at least $\bigOm{ \log n }$. We conclude that
    \begin{equation}\label{eq: ylower}
        \log\abs{y} = \bigOm{ n \, \log n }.
    \end{equation}
    
    This gives a contradiction to Theorem~\ref{thm:bugy}: We have effectively bounded the regulator $R$ of $\KK_n$ in Lemma~\eqref{lem: regulatorapprox} by $(\log n)^2$ and the unit rank is $r = 2$. Further, by Lemma~\ref{lem: lapprox}, an effective upper bound to the coefficients of $f_{n, s, t}$ is given by $n^{2 \max\set{\abs{s}, \abs{t}}}$. If we plug everything into Theorem~\ref{thm:bugy} and take the logarithm, we get 
    \begin{align*}
        \log\abs{y} &= \bigO{ \br{ \log n }^2 \log\log n \br{ \br{\log n}^2 + \max\set{\abs{s}, \abs{t}} \log n } } \\
        &= \bigO{ n^{\frac{1}{2} - \varepsilon} \br{ \log n }^3 \log\log n },
    \end{align*}
    which contradicts Equation~\eqref{eq: ylower}. So provided that $n \geq n_0$ is sufficiently large, the assumption of a solution with $\abs{y} \geq 2$ is false; we thus conclude the proof of Theorem~\ref{thm: main}.

    \bibliography{references}

\end{document}